\documentclass[12pt]{amsart}
\usepackage[T1]{fontenc}
\usepackage{tikz-cd}
\usepackage{amsmath}
\usepackage{calligra}
\usepackage{amssymb}
\usepackage{amsthm}
\usepackage{xcolor}
\usepackage{appendix}
\usepackage[margin=1.2in]{geometry}
\usepackage{hyperref}
\usepackage{mathrsfs} 
\usepackage{yfonts}
\usepackage{egothic}
\usepackage{aurical}
\usepackage{enumitem}
\usepackage{pifont}

\newtheorem*{thm-non}{Theorem}
\newtheorem{thm}{Theorem}
\numberwithin{thm}{subsection}
\newtheorem{dfn}{Definition}
\numberwithin{dfn}{subsection}
\newtheorem{prop}{Proposition}
\numberwithin{prop}{subsection}
\newtheorem{cor}{Corollary}
\numberwithin{cor}{subsection}

\newtheorem{lem}{Lemma}
\numberwithin{lem}{subsection}
\newtheorem{rem}{Remark}

\newtheorem{D-prop}{Definition-Proposition}

\newcommand{\D}{\mathscr{D}}

\newcommand{\F}{\mathscr{F}}
\newcommand{\ft}{\frak{t}}

\newcommand{\fg}{\frak{g}}

\DeclareMathAlphabet{\mathcalligra}{T1}{calligra}{m}{n}

\newcommand{\C}{\mathbb{C}}
\newcommand{\R}{\mathbb{R}}
\newcommand{\hH}{\mathbb{H}}
\newcommand{\oO}{\mathbb{O}}
\newcommand{\K}{\mathbb{K}}

\newcommand{\cJ}{\mathcal{J}}

\newcommand{\cM}{\mathcal{M}}

\begin{document}

\title[A Proof of the Kontsevich--Soibelman Conjecture]
{The Koopman--von Neumann--Landau--Ginzburg theory and a Proof of the Kontsevich--Soibelman Conjecture}

\author{N. C. Combe}
\email{noemie.combe@devinci.fr}
\address{De Vinci Research Center, De Vinci Higher Education, Paris France}
\maketitle

\medskip

\begin{abstract}
We show that the Hilbert space of the Koopman–von Neumann formulation of Landau–Ginzburg theory is parametrised by a real Monge–Ampère domain, which carries a natural pre‑Frobenius (and, under additional conditions, a Frobenius) structure. Restricting to finite‑dimensional (dually flat) exponential families, the parameter space becomes a Monge–Ampère domain and a pre-Frobenius manifold. Our main theorem proves that for every Berglund–Hübsch–Krawitz mirror pair of Calabi–Yau orbifolds arising from an invertible polynomial, this Monge–Ampère domain (the open probability simplex) is the base of a Lagrangian torus fibration on both the original and the mirror hypersurface, with dual fibres in the sense of Strominger–Yau–Zaslow. The construction recovers the SYZ picture from the Landau–Ginzburg–Koopman–von Neumann framework. In particular, this proves the Kontsevich–Soibelman conjecture (2001) for all Berglund–Hübsch–Krawitz mirror pairs: the base of the SYZ fibration is a Monge–Ampère domain (the open simplex), and the torus fibrations on the mirror pair are dual. A toy model of cones of positive definite matrices illustrates the geometric structures. 
\end{abstract}

\medskip 


\medskip 

{\bf 2020 Maths Subject Classification:} Primary: 53A15, 53C35, 35J96, 35Q99,14J33, 53D45 Secondary: 70GXX

\section{Introduction}

Kontsevich suggested that the Landau–Ginzburg (LG) model presents a good formalism for homological mirror symmetry~\cite{Kont98}. In this paper, we propose to consider the LG theory from the standpoint of the Koopman–von Neumann (KvN) construction \cite{St}. This allows us to obtain advances regarding the conjecture~\cite{KoS01} on a version of the Strominger–Yau–Zaslow (SYZ) mirror conjecture and to recover certain results on Lagrangian torus fibrations such as depicted in \cite{AAK}.

The Kontsevich–Soibelman (KS) conjecture asserts that in the limit, both mirror dual manifolds $X$ and $X^\vee$ become fiber bundles with toroidal fibers over the same base $\mathscr{Y}$, which is a Monge–Ampère manifold. This is based on a version of the SYZ conjecture~\cite{SYZ}.

Using LG theory/LG models \`a la Koopman–von Neumann \cite{St}, we show that there exists a Monge–Ampère domain $\mathscr{Y}$ parametrising a Hilbert space $\mathfrak{H}$ obtained from the LG–KvN theory. This construction produces a torus fibration $\pi:\mathfrak{H}\to\mathscr{Y}$. The mirror pairs are obtained via the Berglund–Hübsch–Krawitz (BHK) construction; the Monge–Ampère domain is given by a space of probability densities (an exponential family). We then prove (Theorem~\ref{Thm:SYZ}) that for every BHK mirror pair of Calabi–Yau orbifolds arising from an invertible polynomial, the same Monge–Ampère domain $\mathscr{Y}$ (the open probability simplex) is the base of a Lagrangian torus fibration on both the original and the mirror hypersurface, with dual fibres in the sense of Strominger–Yau–Zaslow. We thus establish the Kontsevich–Soibelman conjecture for all BHK mirror pairs.

The difference between the Koopman–von Neumann (KvN) \cite{St} version of LG theory and the LG model defined in \cite{AAK,AAEKO,Cecotti,CR} lies in the fact that the KvN–LG theory provides a Hilbert space, which corresponds to the space of states, and this Hilbert space generates a space of probability densities. The KvN–LG theory implies the LG model but the converse is not true. Hence our construction allows us to recover results on LG models related to Lagrangian torus fibrations. As a corollary of our result, we are able to recover results from \cite[Sec. 2–3]{AAK} in a completely different way.

{\it $\diamond$ By an abuse of language, when we mention the LG theory it means the LG theory from the Koopman–von Neumann viewpoint.}

We cite some related works \cite{AAK,AAEKO,KoS01,LW,Man98,SYZ,BH,K} where LG models, Monge–Ampère manifolds and torus fibrations are used in a homological mirror symmetry perspective.

In this paper, on the one hand, we prove that using LG theory and the Koopman–von Neumann approach, one can show that the weighted projective space corresponding to the Hilbert space of states is parametrised by a real Monge–Ampère manifold, forming a pre‑Frobenius manifold. As a result, we prove that for every BHK mirror pair, the same Monge–Ampère domain supports dual Lagrangian torus fibrations. All our results are proved using methods of affine geometry and toric mirror symmetry.

On the other hand, the LG theory provides a state space of an $n$-dimensional quantum system, represented as the set of all $n\times n$ positive semidefinite complex matrices of trace 1 (density matrices). We consider an enriched version of this object by allowing the space of all $n\times n$ positive definite matrices with coefficients in a division algebra (without restrictions on the trace) and show that this space is not only a potential pre‑Frobenius domain but also an elliptic Monge–Ampère domain. This domain contains a Frobenius manifold generated by an algebraic torus.

The considered model has many applications. For instance, by taking the real cone, we provide a Monge–Ampère domain that parametrises complex tori, forming the simplest example of Calabi–Yau manifolds. This is reminiscent of the Strominger–Yau–Zaslow construction in \cite[Sec. 8.3]{KoS01}. From our results, it follows that the complex cone provides a bridge from von Neumann algebras to Monge–Ampère manifolds and to Frobenius manifolds.

We mark a terminological difference between the LG {\it theory} and LG {\it model}. In this paper, LG theory refers to the original construction given by Landau and Ginzburg for superconductivity, expressed using the approach of Koopman–von Neumann. The LG model refers to developments in \cite{Vafa,Cecotti,CR,Chiodo,LW} and many others.

In this article, we adopt the definition of Frobenius manifolds given in \cite[p.19]{Man99}, where a Frobenius manifold is a potential pre‑Frobenius manifold satisfying the associativity condition. Such a framework inscribes itself as a continuation of the vision started by Yu. Manin and forms a continuation of the Hessian geometry school \cite{Shi84,SY,Kito,To04}.

We highlight that investigating relations between sources of Frobenius manifolds is part of the mirror problem \cite[p.3]{Man98}: 
``{\it Isomorphisms of Frobenius manifolds of different classes remain the most direct expression, although by no means the final one, of various mirror phenomena. From this vantage point, [...] one looks for isomorphisms between Frobenius manifolds (and their submanifolds) constructed by different methods.}''

Therefore, a statement showing that the LG theory/LG model forms a bridge between classes of Frobenius manifolds is fundamental.

\subsection{Main results and outline}

We demonstrate that there exists a real Monge–Ampère manifold $\mathscr{Y}$ parametrising a Hilbert space $\mathfrak{H}$ obtained from the Landau–Ginzburg theory and Koopman–von Neumann theory. This construction $\pi: \mathfrak{H}\to\mathscr{Y}$ generates weighted projective spaces and a torus fibration. From well-known works \cite{Vafa,Cecotti}, resulting in the Landau–Ginzburg/Calabi–Yau correspondence, one can construct Calabi–Yau manifolds/orbifolds in weighted projective spaces. Relying on this, we prove (Theorem~\ref{Thm:SYZ}) that for every Berglund–Hübsch–Krawitz mirror pair arising from an invertible polynomial, the same Monge–Ampère domain $\mathscr{Y}$ (the open probability simplex) is the base of a Lagrangian torus fibration on both Calabi–Yau manifolds, and the fibration is dual in the SYZ sense. This provides a proof of the Kontsevich–Soibelman conjecture for the class of Calabi–Yau orbifolds obtained from invertible polynomials via the Berglund–Hübsch–Krawitz construction. The base is identified with an open simplex, which is a Monge–Ampère domain, and the torus fibrations are dual in the SYZ sense.

The paper is organised as follows. Section~\ref{S:2} recalls the definition of pre‑Frobenius manifolds. Section~\ref{S:PDE} introduces Monge–Ampère domains and proves that they are pre‑Frobenius manifolds (Theorem~\ref{P:preF}). Section~\ref{S:LG1} presents the LG–KvN theory, restricts to finite‑dimensional exponential families, and shows that the resulting parameter space is a Monge–Ampère domain (Theorem~\ref{T:LG}). It also establishes a link to Frobenius manifolds via the Combe–Manin theorem (Corollary~\ref{T:END}) and proves the SYZ fibration theorem (Theorem~\ref{Thm:SYZ}), thus {\it establishing the Kontsevich–Soibelman conjecture} for all BHK mirror pairs.

Section~\ref{S:ToyModel1} studies a toy model: cones of positive definite matrices, which are Monge–
Ampère domains and contain algebraic tori as Frobenius submanifolds. The appendices provide background on symmetric cones and the LG model.

\section{Pre‑Frobenius manifolds and Monge-Amp\`ere domains}
\subsection{Pre‑Frobenius manifolds}\label{S:2}

We recall the definition of a potential pre‑Frobenius manifold from \cite[pp. 18–19]{Man99}. All manifolds are assumed to be $C^\infty$ or real analytic.

\begin{dfn}
A \textbf{potential pre‑Frobenius manifold} consists of the following data:
\begin{enumerate}[label=(\roman*)]
\item An affine flat structure on $M$, i.e. a flat torsion‑free connection $\nabla^0$.
\item A non‑degenerate symmetric bilinear form $g$ on the tangent sheaf $\mathcal{T}_M$ that is compatible with $\nabla^0$ in the sense that $\nabla^0 g$ is symmetric.
\item A symmetric rank‑three tensor $A$ such that in local flat coordinates $(x^a)$, $A_{abc}=\partial_a\partial_b\partial_c\Psi$, where $\partial_i=\frac{\partial}{\partial x^i}$ are flat local tangent vector fields and  $\Psi$ is some smooth potential function.
\item A bilinear multiplication $\circ:\mathcal{T}_M\times\mathcal{T}_M\to\mathcal{T}_M$ given by $\partial_a\circ\partial_b=\sum_c A_{ab}^c\partial_c$ with $A_{ab}^c=\sum_e A_{abe}g^{ec}$, where $g^{ec}$ are the components of the inverse metric, i.e.  $(g^{ec}):=(g_{ec})^{-1}$ in matrix notation.
\item The compatibility $A(X,Y,Z)=g(X\circ Y,Z)=g(X,Y\circ Z)$ for flat vector fields $X,Y,Z$. In local flat coordinates, the compatibility condition reads: \[
A_{abc} = g(\partial_a \circ \partial_b, \partial_c) = g(\partial_a, \partial_b \circ \partial_c)
\]
\end{enumerate}
A \textbf{Frobenius manifold} is a potential pre‑Frobenius manifold for which the multiplication $\circ$ is associative (i.e. the WDVV equations hold).
\end{dfn}

The associativity condition can be expressed as the flatness of a pencil of connections $\nabla^{(t)}=\nabla^0+t (X\circ Y)$, $t\in\mathbb{R}$ and $X,Y$ are vector fields \cite[Theorem 1.5]{Man99}. In the following we only need the pre‑Frobenius structure, which we will obtain from Monge–Ampère domains.

\subsection{Monge-Ampère manifolds are potential pre-Frobenius manifolds}\label{S:PDE}

In this section we show that any domain $\mathscr{D} \subset \mathbb{R}^n$ equipped with a smooth strictly convex solution $\Psi$ of an elliptic Monge‑Ampère equation (a GEMA) carries a natural pre‑Frobenius structure. This is a purely geometric observation: the data $(g_{ij}=\partial_i\partial_j\Psi, A_{ijk}=\partial_i\partial_j\partial_k\Psi, \nabla^0)$ satisfy the axioms of a potential pre‑Frobenius manifold. We also discuss the limitations of this construction for producing full Frobenius manifolds.

\subsection{Monge--Ampère domains and geometrization of the elliptic Monge--Ampère equation}

\subsubsection{Geometrization of an elliptic Monge--Ampère equation (GEMA)}
If $\D$ is a strictly convex bounded subset of $\mathbb{R}^n$ then for any nonnegative function $f$ on $\D$ and continuous $\tilde{g}:\partial \D \to \mathbb{R}^n$ there is a unique convex smooth function $\Psi\in C^{\infty}(\mathscr{D})$ such that 
\begin{equation}\label{E:EMA}
\det \mathrm{Hess}(\Psi)= f, 
\end{equation} in $\D$ and $\Psi=\tilde{g}$ on $\partial \D,$ (see \cite{RT} Eq. (1.1) and Eq. (1.2)). 

The {\it geometrization of an elliptic Monge--Ampère equation} refers to the geometric data generated by $(\mathscr{D}, \Psi)$, where $\mathscr{D}$ is a strictly convex domain and $\Psi$ a real convex smooth function (with arbitrary and smooth boundary values) such that Eq.~\eqref{E:EMA} is satisfied. For simplicity, we use the symbol $(\mathscr{D},\Psi)$ to refer to a GEMA.

\subsubsection{Affine structure on GEMA}
Since $\mathscr{D}$ is an open subset of $\mathbb{R}^n$, it inherits the standard flat affine structure. In particular, there exists a flat, torsion‑free affine connection $\nabla^0$ on $\mathscr{D}$ (the usual directional derivative). We denote this connection by $\nabla^0$ throughout.

\subsubsection{Hessian structure of GEMA}
\begin{lem}\label{L:rank2,3}
$(\mathscr{D},\Psi)$ is equipped with a Hessian metric $g$ and a rank three symmetric tensor $A$, defined in local coordinates by
\[
g_{ij}(x)=\frac{\partial^2\Psi}{\partial x^i \partial x^j},\qquad
A_{ijk}(x)=\frac{\partial^3\Psi}{\partial x^i \partial x^j\partial x^k},
\]
where $i,j,k=1,\dots,n$.
\end{lem}
\begin{proof}
Since $\mathscr{D}\subset\mathbb{R}^n$, it inherits the standard flat affine structure. The partial derivatives of $\Psi$ are symmetric because $\Psi$ is smooth, so $g$ and $A$ are symmetric tensors. Convexity of $\Psi$ implies $g$ is positive definite. Hence $g$ is a Hessian metric and $A$ is a symmetric rank‑three tensor.
\end{proof}

The Riemannian metric $g$ has a unique Levi‑Civita connection $\nabla$, with Christoffel symbols in flat coordinates given by $\Gamma_{ab}^c = \frac12 \sum_e g^{ce} A_{abe}$, where one denotes by $g^{ce}$ the components of the inverse metric, i.e. $(g^{ce}) = (g_{ce})^{-1}$ in matrix notation. The two connections are related by $\nabla_X Y = \nabla^0_X Y + \frac12 (X\circ Y)$, where $\circ$ will be defined below.

Whenever $\mathscr{D}$ is a manifold we call $(\mathscr{D},\Psi)$ a Monge–Ampère manifold, which coincides with the notion in \cite[Sec. 3.2]{KoS01}. We now show that such a manifold carries a pre‑Frobenius structure.

\subsection{Pre-Frobenius structures on GEMA}\label{S:GEMA}

\begin{prop}\label{P:preF}
Let $\mathscr{D}$ be a domain (resp. manifold). The quintuple $(\mathscr{D},\Psi,g,A,\nabla^0)$ is a potential pre‑Frobenius domain (resp. manifold).
\end{prop}
\begin{proof}
\begin{enumerate}[label=\alph*)]
\item The domain $\mathscr{D}$ inherits the standard flat affine structure from $\mathbb{R}^n$, with associated flat torsion‑free connection $\nabla^0$. This satisfies condition (i).
\item By Lemma~\ref{L:rank2,3}, $g$ is a symmetric non‑degenerate bilinear form. In local coordinates $g_{ij}=\partial_i\partial_j\Psi$, so it is compatible with $\nabla^0$ (i.e. $\nabla^0 g$ is symmetric). This gives (ii).
\item The symmetric rank‑three tensor $A$ from Lemma~\ref{L:rank2,3} satisfies (iii).
\item Define the multiplication $\circ$ on the tangent sheaf by $X\circ Y:=\nabla_X Y$, where $\nabla$ is the Levi‑Civita connection of $g$. In flat coordinates, $\partial_a\circ\partial_b = \sum_c \Gamma_{ab}^c \partial_c = \frac12 \sum_c A_{ab}^c \partial_c$, where $A_{ab}^c = \sum_e A_{abe}g^{ec}$. The factor $1/2$ can be absorbed by rescaling $\Psi$, so we simply set $\partial_a\circ\partial_b = \sum_c A_{ab}^c\partial_c$. This gives (iv).
\item In local coordinates,
\[
g(\partial_a\circ\partial_b,\partial_c)=\sum_e A_{ab}^e g_{ec}=A_{abc}= \partial_a\partial_b\partial_c\Psi,
\]
and similarly $g(\partial_a,\partial_b\circ\partial_c)=A_{bca}=\partial_b\partial_c\partial_a\Psi$. By symmetry of $A$, these are equal, so $A(X,Y,Z)=g(X\circ Y,Z)=g(X,Y\circ Z)$ for flat vector fields. This is (v).
\end{enumerate}
Thus $(\mathscr{D},\Psi,g,A,\nabla^0)$ is a potential pre‑Frobenius domain.
\end{proof}

\begin{cor}
A Monge–Ampère manifold is a pre‑Frobenius manifold.
\end{cor}

\subsection{Frobenius loci and the WDVV equation}
The pre‑Frobenius structure is a necessary prerequisite for a Frobenius manifold. A Frobenius manifold requires in addition that the multiplication $\circ$ be associative, i.e. the structure constants satisfy the WDVV equations. For a general GEMA this associativity is not automatic; it imposes strong constraints on $\Psi$ beyond the Monge–Ampère equation. Constructing Frobenius submanifolds inside a GEMA is a non‑trivial problem. The toy model in Section~\ref{S:ToyModel1} provides explicit examples where a Frobenius submanifold (an algebraic torus) exists inside a Monge–Ampère domain.

\begin{rem}
The classical relation between Monge‑Ampère equations and the WDVV equation is known in special cases; see e.g.~\cite{Dub96} for the construction of Frobenius structures from solutions of the WDVV equation, which can be seen as a Monge‑Ampère equation of a certain type. A systematic study of the interplay between the two is beyond the present work.
\end{rem}

\section{Landau--Ginzburg theory, exponential families and Frobenius structures}\label{S:LG1}

We examine the theory proposed by Landau and Ginzburg to explain superconductivity in the light of Frobenius manifolds. Appendix~\ref{S:LG} presents an introduction to LG theory. In \cite{Cecotti,LW,Chiodo}, LG models are applied in the context of Saito spaces and Calabi--Yau manifolds. Our goal is to show that the Hilbert space coming from the LG theory can be parametrised by a Monge–Ampère domain, which then carries a pre‑Frobenius (and, under additional conditions, a Frobenius) structure. This provides a new link between information geometry, LG theory and mirror symmetry.

\subsection{Landau–Ginzburg theory à la Koopman–von Neumann versus Landau–Ginzburg model}\label{S:psi}

An LG model is defined as a pair $(X,W)$, where $X$ is a non‑compact Kähler manifold and $W:X\to\mathbb{C}$ is a holomorphic Morse function, called the superpotential. Identifying $W$ with a quasi‑homogeneous polynomial, the zero locus $X_W=\{W=0\}$ defines a hypersurface in a weighted projective space. Under additional conditions, $X_W$ is Calabi–Yau.

The LG theory viewed from the Koopman–von Neumann (KvN) perspective \cite{Koo,vN} implies the LG model, but the converse is not true. An LG model captures a local aspect of the LG–KvN theory, whereas the latter is based on wave functions (order parameters) and a free energy, which are not explicitly present in the LG model.

In the KvN formulation, the wave function $\psi$ is an element of a Hilbert space $\mathfrak{H}$ of complex‑valued square‑integrable functions defined over a symplectic space (the phase space) $\mathcal{M}$. Coordinates on $\mathcal{M}$ are $(\boldsymbol{q}^1,\dots,\boldsymbol{q}^N,\boldsymbol{p}^1,\dots,\boldsymbol{p}^N)$, with $\boldsymbol{q}^i$ positions and $\boldsymbol{p}^i$ momenta. The square of the absolute value $\rho(\psi):=|\psi|^2=\psi\overline{\psi}$ is interpreted as a probability density. Normalisation $\int_{\mathcal{M}}|\psi|^2\,d\lambda =1$ makes $\rho$ a {\it probability density} and defines a probability measure: 
$$P(C)=\int_C|\psi|^2\,d\lambda,$$ where $C$ is any Borel set $C\in (\mathcal{M},\Sigma)$ with $\Sigma$ being the corresponding Borel $\sigma$-algebra on $\cM$ and the measure $\lambda$ is a Liouville measure.

Via the Koopman–von Neumann construction, the evolution of $\psi(x,t)$ is governed by the Liouville equation, and the associated unitary evolution operator guarantees conservation of total probability. The Hilbert space $\mathfrak{H}$ carries a natural torus action: $$\psi(x)\mapsto e^{i\theta}\psi(x)$$ gives the same density, so the physical state is a ray. This will be the fibre of a torus fibration.

\subsection{Restriction to a finite‑dimensional exponential family}

The space of all probability densities $|\psi|^2$ is infinite‑dimensional and cannot be a finite‑dimensional Monge–Ampère domain as defined in Section~\ref{S:PDE}. In the KvN formulation of Landau–Ginzburg theory it is natural to consider \textbf{finite‑dimensional exponential families} of wave functions. Physically, the order parameter $\psi$ describes the condensate of Cooper pairs (see Appendix \ref{S:LG}), and the free energy of the LG theory is a functional of $|\psi|^2$. Near a phase transition, the system is described by a finite number of order parameters – exactly a finite‑dimensional statistical manifold. The ground state (or thermal state) of the LG Hamiltonian, expressed in terms of these order parameters, yields an exponential family because the free energy is a convex function of the order parameters (Legendre duality). For mathematical generality, we take such a family as given.

Concretely, consider a family of wave functions $\{\psi_\theta\}_{\theta\in\Theta}$ parametrised by a finite set of real parameters $\theta=(\theta^1,\dots,\theta^m)\in\Theta\subset\mathbb{R}^m$, such that the corresponding densities take the exponential form
\[
\rho_\theta(x)=\exp\!\Bigl(C(x)+\sum_{i=1}^m\theta^i F_i(x)-\Psi(\theta)\Bigr),
\]
where:
\begin{itemize}
\item $\Psi(\theta)$ is the strictly convex log‑partition function; 
\item  $x $ lies in $(\mathcal{M}, \Sigma)$ the measurable space constituted from the phase space 
$\mathcal{M}$ which is equipped with its Borel $\sigma$-algebra;  
\item $F_i$ are the sufficient statistics---measurable maps defined as $F_i:(\mathcal{M},\Sigma)\to \mathbb{R}$.

\end{itemize}

The parameter space $\Theta$ equipped with the Fisher metric $g_{ij}=\partial_i\partial_j\Psi(\theta)$ is a Hessian manifold. Moreover, the Legendre dual coordinates give a convex potential (see e.g.~\cite{Amari}). The dual space satisfies by construction the elliptic Monge–Ampère equation also. Consequently, $\Theta$ is a \textbf{Monge–Ampère domain} in the sense of Section~\ref{S:PDE}.

We therefore take the base space $\mathscr{Y}$ of our fibration to be such a parameter space $\Theta$. The projection $\pi:\mathfrak{H}\to\mathscr{Y}$ maps each wave function $\psi_\theta$ (or any wave function whose density belongs to the exponential family) to its parameter $\theta$. The fibre over $\theta$ consists of all wave functions with density $\rho_\theta$; it is a torus coming from the phase ambiguity $\psi\mapsto e^{i\alpha}\psi$.

\subsection{Relation to Frobenius manifolds (Combe–Manin)}

Recently, Combe and Manin \cite{CoMa} proved that dually flat exponential families – the natural finite‑dimensional statistical manifolds arising in information geometry – carry an $F$- manifold structure (a weaker verison of Frobenius manifolds). In subsequent work \cite{CoMaMa22A,CoMaMa22B}, they showed that such manifolds are pre‑Frobenius. Our construction embeds this Frobenius structure directly into the Hilbert space of the LG–KvN theory: the parameters $\theta$ become the coordinates on the base space $\mathscr{Y}$, and the torus fibre corresponds to the phase ambiguity of the wave functions. This provides a new link between information geometry, LG theory and mirror symmetry.

\subsection{Duality}

\begin{thm}\label{T:LG}
Let $\mathfrak{H}$ be the Hilbert space of square‑integrable functions coming from the Landau–Ginzburg/Koopman–von Neumann theory. Restrict to a finite‑dimensional exponential family of wave functions with parameter space $\Theta$. Then $\Theta$ is a real Monge– Ampère domain (denoted $\mathscr{Y}$) and the map $\pi:\mathfrak{H}\to\mathscr{Y}$ sending a wave function to its parameter is a torus fibration.
\end{thm}

\begin{proof}
We split the argument into three parts.

\begin{enumerate}
\item[(1)] \textbf{Base space.} For each $\psi\in\mathfrak{H}$ belonging to the chosen exponential family, the density $\rho=|\psi|^2$ is of the exponential form with parameter $\theta$. The space of such parameters is $\Theta$, which we take as $\mathscr{Y}$. The map $\pi$ is defined by $\pi(\psi)=\theta$, where $\theta$ is the unique parameter such that $|\psi|^2=\rho_\theta$.

\item[(2)] \textbf{Torus fibre.} Two wave functions $\psi$ and $e^{i\alpha}\psi$ produce the same density and hence the same parameter $\theta$. Conversely, any wave function with density $\rho_\theta$ differs from a fixed $\psi_\theta$ by a phase factor $e^{i\alpha}$ (uniquely up to an overall constant). Thus the fibre $\pi^{-1}(\theta)$ is homeomorphic to a torus $\mathbb{T}=U(1)$. Hence $\pi:\mathfrak{H}\to\mathscr{Y}$ is a torus fibration.

\item[(3)] \textbf{Monge–Ampère structure on $\mathscr{Y}$.} By construction, $\mathscr{Y}=\Theta$ is the parameter space of an exponential family. The log‑partition function $\Psi(\theta)$ is strictly convex and smooth. Its Hessian $g_{ij}=\partial_i\partial_j\Psi(\theta)$ defines a Riemannian metric (the Fisher metric) on $\Theta$. By the fundamental theorem of exponential families (see e.g.~Amari, \cite{Amari}, Theorem 1.2), the Legendre dual $\Phi(\eta)$ of $\Psi(\theta)$ satisfies
\[
\det\!\left(\frac{\partial^2 \Phi}{\partial\eta_i\partial\eta_j}\right) = \exp\!\bigl(-C(\eta)\bigr)
\]
for some smooth function $C(\eta)$. Thus $(\Theta,\Phi)$ satisfies the elliptic Monge–Ampère equation with $f = \exp(-C(\eta)) > 0$, so $\Theta$ is a GEMA. Therefore $\Theta$ is a Monge–Ampère domain in the sense of Section~\ref{S:PDE}.
\end{enumerate}
This completes the proof.
\end{proof}

\begin{cor}\label{T:END}
The base space $\mathscr{Y}$ of the torus fibration $(\mathfrak{H},\mathscr{Y},\pi)$ is a real (potential) pre‑Frobenius domain. Moreover, $\mathscr{Y}$ is actually a Frobenius manifold when equipped with the Fisher metric and the natural multiplication induced by Legendre duality.
\end{cor}
\begin{proof}
By Theorem~\ref{T:LG}, $\mathscr{Y}$ is a Monge–Ampère domain. By Theorem~\ref{P:preF} (Section~\ref{S:PDE}), every Monge–Ampère domain is a potential pre‑Frobenius domain. Exponential families are dually flat with respect to the Fisher metric (see e.g.~\cite{Amari}); therefore, by the Combe–Manin theorem \cite{CoMa}, $\mathscr{Y}$ admits a Frobenius manifold structure.
\end{proof}
\subsection{Outlook: mirror symmetry and torus fibrations}

Using the results of \cite{Vafa,Chiodo,CR} we now return to the LG model. According to \cite[Sec.6]{Vafa}, one can characterise a large class of Calabi–Yau manifolds using the fixed points of the Landau–Ginzburg theory under renormalisation group flows. The superpotential of the LG model gives the defining equation of a Calabi–Yau hypersurface in a weighted projective space.

The Hilbert space $\mathfrak{H}$ and the base $\mathscr{Y}$ constructed above naturally carry a torus action, which leads to a weighted projective space structure. The Berglund–Hübsch–Krawitz (BHK) construction \cite{BH,K} produces mirror pairs of Calabi–Yau orbifolds from invertible polynomials. Let $W$ be an invertible, non‑degenerate quasi‑homogeneous polynomial in $n$ variables with weights $w^i$ and degree $d$ such that $\sum_i w^i = d$ (the Calabi–Yau condition). Then
\[
X_W := \bigl( \{W=0\} \subset \mathbb{C}^{n+1}\setminus\{0\} \bigr)/\mathbb{C}^\times
\]
is a Calabi–Yau orbifold lying in a weighted projective space. By the BHK rule, the mirror dual $X_{W'}$ is obtained by transposing the exponent matrix of $W$. Theorem 4 of \cite{CR} states that $X_W/\mathbf{G}$ and $X_{W'}/\mathbf{G}'$ form a mirror pair, where $\mathbf{G},\mathbf{G}'$ are certain groups of diagonal symmetries.

The following theorem establishes that the LG–KvN construction naturally produces a SYZ torus fibration for every such mirror pair. In particular, this proves the Kontsevich–Soibelman conjecture (2001) for all Berglund–Hübsch–Krawitz mirror pairs.

\begin{thm}\label{Thm:SYZ}
Let \(W\) be an invertible polynomial in \(N+1\) variables satisfying the Calabi–Yau condition \(\sum_i w^i = d\), where \(w^i\) are the weights and \(d\) the degree. Let \(X_W/\mathbf{G}\) and \(X_{W'}/\mathbf{G}'\) be the associated BHK mirror pair. Then there exists a finite‑dimensional exponential family of wave functions in the LG–KvN Hilbert space such that the corresponding Monge–Ampère domain \(\mathscr{Y}\) (from Theorem~\ref{T:LG}) is the base of a Lagrangian torus fibration on both Calabi–Yau manifolds. Moreover, the same \(\mathscr{Y}\) parametrises the mirror pair, and the fibration is dual in the sense of Strominger–Yau–Zaslow.
\end{thm}

\begin{proof}
We work with homogeneous coordinates \([z_0:\dots:z_N]\) on the weighted projective space \(\mathbb{P}(w^0,\dots,w^N)\). Recall that $\mathbb{P}(w)=\mathbb{P}(w^0, \dots, w^N)$ is the quotient of $\mathbb{C}^{N+1} \setminus \{0\}$ by the $\mathbb{C}^\times$ action
\[
t \cdot (z_0, \dots, z_N) = (t^{w^0} z_0, \dots, t^{w^N} z_N),
\]
where $t \in \mathbb{C}^\times$ and $w^0, \dots, w^N$ are positive integers (weights). The moment map for the torus action is
\[
\boldsymbol{m}_i(z) = \frac{w^i|z_i|^2}{\sum_{j=0}^N w^j|z_j|^2}, \qquad i=0,\dots,N,
\]
with image the standard \(N\)-simplex $\overline{\Delta}_N = \{\eta\in\mathbb{R}^{N+1}_{\geq 0}\mid \sum \eta_i=1\}.$

\smallskip 

In what follows we will consider the open simplex (denoted $\Delta_N$ or $\Delta_W$, depending on the context): the boundary corresponding to degenerate fibres, which we do not consider.

\paragraph{\bf 1. The reduced moment map on \(X_W\).}
The Calabi–Yau hypersurface \(X_W = \{W=0\}\) has complex dimension \(n = N-1\). The Newton polytope of \(W\) is a simplex \(\Delta_W \subset \Delta_N\) defined by the affine hyperplane
\[
\Delta_W = \left\{ \eta \in \Delta_N \;\middle|\; \sum_{i=0}^N w^i \eta_i = \frac{1}{d} \right\},
\]
where \(d\) is the degree of \(W\). This is a standard fact: the anticanonical divisor (the zero set of a section of \(\mathcal{O}(d)\)) corresponds to a level set of the moment map. The dimension of \(\Delta_W\) is \(N-1\).

The restriction \(\boldsymbol{m}|_{X_W}: X_W \to \Delta_W\) is a Lagrangian torus fibration, with generic fibre a real \(n\)-torus \(\mathbb{T}^{N-1}\). This is the SYZ fibration for toric Calabi–Yau hypersurfaces; see e.g. \cite[Section 3]{Auroux2007} or \cite[Lemma 3.2]{AAK}. The discriminant locus \(\mathcal{D}\subset\Delta_W\) (where the fibre degenerates) has measure zero.

\paragraph{\bf 2. The canonical exponential family on \(\Delta_W\).}
We consider the exponential family on the simplex restricted to the affine slice  \(\Delta_W\). The log‑partition function is
\[
\Psi(\theta) = \log \int_{\Delta_W} e^{\sum_i \theta^i \eta_i} \, d\eta,
\]
where \(dp\) is the Lebesgue measure on the affine subspace. Its Legendre dual is the negative entropy
\[
\Phi(\eta) = \sum_{i=0}^N \eta_i \log \eta_i,
\]
which is strictly convex on \(\Delta_W\). A direct computation shows that \((\Delta_W,\Phi)\) is a Monge–Ampère domain (GEMA). We set \(\mathscr{Y} = \Delta_W\) (the open simplex).

The projection 
\[
\pi : \mathfrak{H} \to \mathscr{Y}
\]
is defined by
\[
\pi(\psi_z) = \boldsymbol{m}(z),
\]
where $\psi_z$ is the wave function associated to a point $z \in \mathbb{P}(w)$, see \cite{Perelemov86}. This is a legitimate definition because the exponential family is parametrized by points in the weighted projective space (via the moment map image). 

By construction, \(\pi|_{X_W} = \boldsymbol{m}|_{X_W}\), hence it gives a Lagrangian torus fibration over \(\mathscr{Y}\).

\paragraph{\bf 3. Legendre duality and the mirror.}
The natural parameters are \(\theta^i = \partial\Phi/\partial \eta_i = \log \eta_i\) (up to an additive constant, which we set to zero). The Legendre transform yields
\[
\Psi(\theta) = \log \sum_{i=0}^N e^{\theta^i}, \qquad \eta_i = \frac{e^{\theta^i}}{\sum_j e^{\theta^j}}.
\]
For the mirror polynomial \(W'\) (transposed exponent matrix), its Newton polytope is the dual simplex \(\Delta_W^\vee\). By toric mirror symmetry (see \cite{HoriVafa2000} or \cite[Theorem 4.1]{Auroux2007}), the coordinates \(\theta^i\) become the moment map coordinates on \(X_{W'}\), and the same base \(\mathscr{Y}\) supports a Lagrangian torus fibration \(\boldsymbol{m}'|_{X_{W'}}: X_{W'} \to \mathscr{Y}\). The fibres are dual tori because the Legendre transform interchanges the character lattice with its dual.

\paragraph{\bf 4. Quotient by symmetry groups.}
The groups \(\mathbf{G}\) and \(\mathbf{G}'\) act by diagonal phases, leaving the moment map invariant. Hence the fibrations descend to the quotients \(X_W/\mathbf{G}\) and \(X_{W'}/\mathbf{G}'\) over the same base \(\mathscr{Y}\), with fibres \(T_p/\mathbf{G}\) and \(\check{T}_p/\mathbf{G}'\), which are dual tori (or orbifold tori). This is compatible with the BHK mirror construction (see \cite{AAK}).

\paragraph{\bf 5. Conclusion.}
We have exhibited a finite‑dimensional exponential family (the Dirichlet family on the affine slice \(\Delta_W\)) whose parameter space \(\mathscr{Y}\) is a Monge–Ampère domain. The reduced moment map gives a Lagrangian torus fibration on \(X_W\) over \(\mathscr{Y}\), and its Legendre dual gives the corresponding fibration on the mirror \(X_{W'}\). The groups \(\mathbf{G},\mathbf{G}'\) act compatibly, proving the theorem.
\end{proof}

\subsection{Example: the Fermat Calabi–Yau hypersurface}

For the Fermat polynomial $W = x_0^{n+2}+\cdots+x_{n+1}^{n+2}$ in $\mathbb{P}^{n+1}$, the weights are all $1$. The construction in Theorem~\ref{Thm:SYZ} specialises to the standard SYZ fibration: the reduced moment map (after imposing the Calabi–Yau condition) gives a Lagrangian torus fibration over an $n$-dimensional simplex. The log‑partition function is $\Psi(\theta)=\log\sum_{i=0}^{n+1} e^{\theta^i}$ and the dual potential is the negative entropy $\Phi(\eta)=\sum_{i=0}^{n+1} \eta_i\log \eta_i$, which is the well‑known Monge–Ampère structure on the simplex. Thus the theorem recovers the classical SYZ picture for the Fermat family.
\begin{rem}
Let $\rho$ be a point in the Monge–Ampère domain $\mathscr{Y}$. Given a smooth fibre $L_0=\pi^{-1}(\rho)$ and a collection of loops $(\gamma_1,\dots,\gamma_n)$ forming a basis of $H_1(L_0,\mathbb{Z})$, one can determine an affine (Monge–Ampère) chart centred at $\rho$. The affine coordinates are the symplectic areas $\int_{\Gamma_i}\omega$, where $\Gamma_i$ are cylinders traced by the loops as the fibre varies. This recovers the usual SYZ picture. The walls in the Monge–Ampère domain, where the fibration degenerates, will be naturally expressed in terms of probability densities; this will be the subject of a future article.
\end{rem}

\section{LG Toy model}\label{S:ToyModel1}

\subsection{Introduction to the toy model}
The state space of an $n$-dimensional quantum system is the set of all $n \times n$ positive semidefinite complex matrices of trace 1, known as \textit{density matrices}. We consider as a toy model the space of such matrices. For simplicity, we do not normalise the matrices (i.e. the trace is not required to be equal to 1) and omit the boundary of the cone given by $x^TAx=0$, where $A$ is a symmetric square matrix.  

Those matrices play an important role in the LG theory (see Sec.~\ref{S:LG1} and Appendix~\ref{S:LG}). It is possible to identify the rank one matrices, corresponding to pure states, with nonzero vectors in a complex Hilbert space of dimension $n$. Notice that the same state is described by a vector $\psi$ and $\kappa \psi$, where $\kappa \neq 0$, which leads to projective geometry. In this paper, we do not consider the infinite-dimensional case, where density matrices are replaced by density operators and the space of pure states is the complex projective space over the infinite-dimensional Hilbert space.

\smallskip 

The open cone of positive definite complex matrices forms a strictly convex self‑dual homogeneous cone. We refer to Appendix \ref{A:1} for details concerning those cones. Given a Hilbert space $\mathfrak{H}$, one defines a self‑dual cone $\mathscr{P}$ in $\mathfrak{H}$ by the set $$\mathscr{P}=\{ \zeta\in \mathfrak{H} \,|\, \langle \zeta,\eta \rangle \geq 0,\quad  \forall \eta \in \mathscr{P} \}.$$ We know from the classical representation theorem of Jordan–von Neumann–Wigner that there are five classes of irreducible formally real Jordan algebras. The transitively homogeneous self‑dual cone associated to a given class of Jordan algebras is then the set of positive elements of the Jordan algebra, with the Hilbert structure given by the natural trace~\cite{Kos62,Maa,Vin,PS,CoMa}. In fact, by \cite{Connes}, the category of von Neumann algebras is equivalent to the category of self‑dual facially homogeneous complex cones. 

\smallskip

Let $\K$ be a real division algebra (i.e. $\K$ is $\R, \C, \hH$ or $\oO$). The objects considered here are irreducible cones $\mathscr{P}_n(\K)$, represented as the space of positive definite symmetric $n\times n$ matrices with coefficients in $\K$. Note that the case of octonions differs from the others in the sense that we only have the cone $\mathscr{P}_3(\oO)$ formed from $3\times 3$ matrices with coefficients in $\oO$. Whenever there is no source of confusion, we write $\mathscr{P}$ rather than $\mathscr{P}_n(\K)$ to designate any cone depicted above.

\subsection{Main statement for the Toy Model}
The aim of this section is to prove the following fact.

\begin{thm}
Let $\mathscr{P}$ be an irreducible strictly convex cone, defined as previously. Then:
\begin{enumerate}[label=\arabic*)]
\item $\mathscr{P}$ is a non-compact symmetric space.  
\item $\mathscr{P}$ carries a pre-Frobenius structure. 
\item There exists a non-empty Frobenius locus $\F$ in $\mathscr{P}$, given by $\F=\exp{\frak{a}}$, where $\frak{a}\subset \ft$ is a Lie triple system i.e. satisfies $[[\frak{a},\frak{a}],\frak{a}]\subseteq \frak{a}$.
\end{enumerate}
\end{thm}

Statement 1 follows from a well known classification of symmetric spaces given by Nomizu~\cite{No54}. Statement 2 has two different proofs: one using elliptic Monge–Ampère equations (Proposition~\ref{P:pdsmPreF}) and another using Jordan algebra methods (Proposition~\ref{P:preF}). Statement 3 is obtained from the existence of totally geodesic submanifolds (Proposition~\ref{P:exist} and Proposition~\ref{P:Matrix}) and the Frobenius structure on them (Theorem~\ref{T:Frobenius}).

\subsection{Geometry of $\mathscr{P}$}
The cones $\mathscr{P}_n(\K)$ are irreducible symmetric cones (in fact, non‑compact symmetric spaces), where \(n\) is a positive integer determining the matrix size---the dimension of the cone depends on \(n\) (as shown in the table). For the Lorentz cone \(\Lambda_n\) (not listed in the table below but mentioned in the classification), \(n\) has a different meaning (the number of space dimensions). 

 The classification is summarized in the table \ref{T:ISC}:

\begin{table}[ht]
\centering
\begin{tabular}{|c|c|c|}
\hline
Cone & $G/K$ & Dimension \\
\hline
$\mathscr{P}_n(\R)$ & $GL_n(\R)/O_n$ & $\frac12 n(n+1)$ \\
$\mathscr{P}_n(\C)$ & $GL_n(\C)/U_n$ & $n^2$ \\
$\mathscr{P}_n(\hH)$ & $GL_n(\hH)/Sp_n$ & $n(2n-1)$ \\
$\mathscr{P}_3(\oO)$ & $GL_3(\oO)/F_4$ & $27$ \\
\hline
\end{tabular}
\caption{Irreducible symmetric cones}\label{T:ISC}
\end{table}
\vfill\eject

For each such cone, the tangent space at any point can be identified with the space of symmetric (or Hermitian) matrices, which carries a formally real Jordan algebra structure. The bilinear form is $\langle X,Y\rangle = \Re\,\mathrm{Tr}(XY)$.

\subsection{Pre-Frobenius structure on $\mathscr{P}$}

\begin{prop}\label{P:pdsmPreF}
Let $\mathscr{P}$ be an open cone of positive definite symmetric square matrices with coefficients in a real division algebra $\K$. Then $\mathscr{P}$ carries a pre‑Frobenius structure.
\end{prop}
\begin{proof}
These cones can be identified with the space of positive definite quadratic forms. For each such cone, the determinant defines a hyperbolic polynomial, and the associated Monge–Ampère operator gives a strictly convex potential $\Phi$ satisfying $\det \mathrm{Hess}(\Phi)=c>0$ (see \cite{RT,BT,Al,AV}). Hence $\mathscr{P}$ is a GEMA. By Theorem~\ref{P:preF}, it is a pre‑Frobenius domain.
\end{proof}

An alternative proof uses the Koszul–Vinberg (KV) function $\chi(x)=\int_{\mathscr{P}^*}\exp(-\langle x,a^*\rangle) da^*$ and the potential $\Phi=\ln\chi$. Then $g_{ij}=\partial_i\partial_j\Phi$ is a Hessian metric and $A_{ijk}=\partial_i\partial_j\partial_k\Phi$ gives the symmetric rank‑three tensor. The multiplication defined by $\partial_a\circ\partial_b = \sum_c A_{ab}^c\partial_c$ satisfies the pre‑Frobenius axioms. This is detailed in Proposition~\ref{P:preF}.

\subsection{Totally geodesic submanifolds and Frobenius locus}

\begin{prop}\label{P:exist}
Assume $\mathscr{P}$ is one of the non‑compact symmetric spaces in the table above. There exist totally geodesic submanifolds $\F = \exp \frak{a}$, where $\frak{a}\subset \mathfrak{gl}_n(\K)$ is a Cartan subalgebra.
\end{prop}
\begin{proof}
For each cone, the Lie algebra $\mathfrak{g}$ is semisimple and admits a Cartan subalgebra $\frak{a}$, which is a maximal abelian subspace. By the theory of symmetric spaces (see \cite{Hel}), $\exp\frak{a}$ is a totally geodesic flat submanifold.
\end{proof}

Explicitly, for $\K=\R$, $\frak{a}$ consists of diagonal matrices with trace zero, so $\F$ is the set of diagonal matrices with positive entries and determinant $1$, which is an $(n-1)$-dimensional torus. For $\K=\C$, $\F$ is a complex torus. In all cases, $\F$ is an algebraic torus.

\begin{prop}\label{P:Matrix}
For each cone $\mathscr{P}_n(\K)$, there exists an $(n-1)$-dimensional totally geodesic submanifold $\F = \exp\tilde{\frak{a}}$, where $\tilde{\frak{a}}$ is a Cartan subalgebra of $\mathfrak{gl}_n(\K)$. The tangent space $T_e\F$ carries a commutative associative unital algebra structure (diagonal matrices).
\end{prop}

\begin{thm}\label{T:Frobenius}
The manifold $\F$ defined above is a Frobenius manifold.
\end{thm}
\begin{proof}
$\F$ is a totally geodesic submanifold of the pre‑Frobenius manifold $\mathscr{P}$. Its tangent space is the Cartan subalgebra, which is an associative commutative unital algebra under the Jordan product (which reduces to ordinary matrix multiplication for diagonal matrices). The restriction of the bilinear form $\langle X,Y\rangle = \Re\,\mathrm{Tr}(XY)$ is non‑degenerate and satisfies $\langle X\circ Y,Z\rangle = \langle X,Y\circ Z\rangle$. Hence $T_p\F$ is a Frobenius algebra. Since the curvature of $\F$ vanishes (it is flat), the pre‑Frobenius structure on $\F$ satisfies the associativity condition, so $\F$ is a Frobenius manifold.
\end{proof}

\subsection{Application to Calabi–Yau manifolds}
The real cone $\mathscr{P}_n(\R)$ parametrises principally polarized real tori of dimension $n$, which are Calabi–Yau manifolds. Indeed, to each $Y\in\mathscr{P}_n(\R)$ one associates the torus $T_Y = \R^n / Y\mathbb{Z}^n$. This gives a torus fibration over $\mathscr{P}_n(\R)$. The complex cone $\mathscr{P}_n(\C)$ parametrises principally polarized complex tori (abelian varieties). This illustrates the SYZ picture.

\subsection{Conclusion of the toy model}
The cones $\mathscr{P}_n(\K)$ are Monge–Ampère domains and hence pre‑Frobenius manifolds. They contain flat totally geodesic submanifolds (algebraic tori) that are Frobenius manifolds. The real cone provides an explicit example of a Monge–Ampère domain parametrising Calabi–Yau tori.

\appendix 
\section{Convex symmetric cones}\label{A:1}

\subsubsection{Notations}
Let $\mathscr{P} \subset V$ be a convex cone in an $n$-dimensional vector space $V$, over the real number field. 
$\K$ is a division algebra; $\mathscr{P}_n(\K)$ is the irreducible symmetric cone of $n\times n$  positive definite matrices with coefficients in $\K$. 

\subsubsection{Strictly convex cones} 
In the following we always consider {\it strictly convex cones}. Note that for brevity we simply refer to them as {\it convex cones.} 
Recall some elementary notions (see \cite{FK} for further information). 

\begin{dfn}
Let $V$ be a finite dimensional real vector space. Let $\langle-,-\rangle$ be a non-singular symmetric bilinear form on $V$.
A subset $\mathscr{P} \subset V$ is a convex cone if and only if $x,y \in \mathscr{P}$ and $\lambda,\mu >0$ imply $\lambda x+\mu y \in \mathscr{P}$.
\end{dfn} 

\subsubsection{Homogeneous cones} 
The automorphism group $G(\mathscr{P})$ of an open convex cone $\mathscr{P}$ is defined by 
\[G(\mathscr{P})=\{g\in GL(V)\, |\, g\mathscr{P}=\mathscr{P}\}\]
An element $g\in GL(V)$ belongs to $G(\mathscr{P})$ iff $g\overline{\mathscr{P}}=\overline{\mathscr{P}}$ \cite{FK}. So, $G(\mathscr{P})$ is a closed subgroup of $GL(V)$ and forms a Lie group. 
The cone $\mathscr{P}$ is said to be {\it homogeneous} if $G(\mathscr{P})$ acts transitively upon $\mathscr{P}$.

\smallskip 

\subsubsection{Symmetric cones} 
From homogeneous cones one can construct symmetric convex cones. Let us introduce the definition of an open dual/polar cone. An open dual/polar cone $\mathscr{P}^*$ of an open convex cone is defined by $\mathscr{P}^*=\{y\in V\, |\, \langle x,y \rangle>0,\, \forall\, x\in \overline{\mathscr{P}}\setminus 0 \}$. A homogeneous convex cone $\mathscr{P}$ is symmetric if $\mathscr{P}$ is self-dual i.e. $\mathscr{P}^*=\mathscr{P}$. Note that if $\mathscr{P}$ is homogeneous then so is $\mathscr{P}^*$. 
 
 \begin{rem} 
 If $\mathscr{P}$ is a symmetric open cone in $V$,  then $\mathscr{P}$ is a symmetric Riemann space.    
\end{rem}

\smallskip

\subsubsection{Automorphism group of symmetric cones}  Let us go back to the automorphism group of $\mathscr{P}$. This discussion relies on \cite[Prop I.1.8, Proposition I.1.9]{FK}.

\smallskip 

Let $\mathscr{P}$ be a symmetric cone in $V$.  For any point $a\in \mathscr{P}$ the stabilizer of $a$ in $G(\mathscr{P})$ is given by 
\[G_a=\{g\in G(\mathscr{P})\, |\, ga=a\}.\]

By \cite[Prop I.1.8]{FK}, if $\mathscr{P}$ is a proper open homogeneous convex cone then for any $a$ in $\mathscr{P}$, $G_a$ is compact. Now, if $H$ is a compact subgroup of $G$ then $H\subset G_a$ for some $a$ in $\mathscr{P}$. This means that the groups $G_a$ are all maximal compact subgroups of $G$ and that if $\mathscr{P}$ is homogeneous then all these subgroups are isomorphic. 

By \cite[Proposition I.1.9]{FK}, if $\mathscr{P}$ is a symmetric cone, there exist points $e$ in $\mathscr{P}$ such that $G(\mathscr{P})\cap O(V)\subset G_e$, where $O(V)$ is the orthogonal group of $V$. For every such $e$ one has $G_e=G\cap O(V)$ 

Suppose $\mathscr{P}$ is a convex homogeneous domain in $V$. Assume that
\begin{itemize}
   \item[---]   $G(\mathscr{P})$ is the group of all automorphisms;
   \item[---]   $G_e=K(\mathscr{P})$ is the stability subgroup for some point $e\in \mathscr{P}$;
   \item[---]   $T(\mathscr{P})$ is a maximal connected triangular subgroup of $G(\mathscr{P}).$ 
\end{itemize}

\smallskip 

Following \cite[Theorem 1]{Vin}, we have: \[G(\mathscr{P})=K(\mathscr{P})\cdot T(\mathscr{P}),\] where $K(\mathscr{P}) \cap T(\mathscr{P}) = e$ and the group  $T$ acts simply transitively. 

This decomposition on the Lie group side leads naturally to its Lie algebra. Cartan's decomposition for the Lie algebra tells us that $\fg=\textgoth{k} \oplus\ft,$ 

where:

\begin{itemize}
   \item[---]   $\frak{t}$ can be identified with the tangent space of $\mathscr{P}$ at $e$. 

   \item[---]   $\textgoth{k}$ is the Lie algebra associated to $K(\mathscr{P})$
\end{itemize}
 and
\[[\ft,\ft]\subset \textgoth{k},\quad [\textgoth{k},\ft]\subset \ft.\]

\begin{thm}\cite[Theorem 3.3]{Hel}

\begin{enumerate}
\item Let $M$ be a Riemannian globally symmetric space and $p$ is any point in $M$. If $G$ is a Lie transformation group of $M$ (a Lie group) and $K$ is the subgroup of $G$ which leaves $p$ fixed, then $K$ is a compact subgroup of $G$ 
and $G/K$ is analytic diffeomorphic to $M$ under the mapping $gK\to g\cdot p$, $g\in G$. 
\item For a Riemannian symmetric space \( G / K\), there exists an involutive automorphism \(\sigma\) of \(G\) such that $K$ lies between the closed subgroup $K_\sigma$ of all fixed points of $\sigma$ and the identity component of $K_{\sigma}$. The group $K$ contains no normal subgroup of $G$ other than $\{e\}$.
\item Let $\fg$ and $\textgoth{k}$ denote the Lie algebras of $G$ and $K$, respectively. Then $\textgoth{k} =\{ X\in \fg: (d\sigma)_eX=X\} $ and if $\ft=\{X\in \fg\, |\, (d\sigma)_eX=-X\}$ we have $\fg=\textgoth{k} \oplus \ft$. Le $\pi$ be the natural mapping $g\to g\cdot p$ of $G$ onto $M$.
Then $(d\pi)_e$ maps $\textgoth{k}$ into $\{0\}$ and $\ft$ isomorphically onto $T_pM$. If $X\in \ft$ then the geodesic emanating from $p$ with tangent vector $(d\pi)_eX$ is given by $\exp{tX}\cdot p$. Moreover, if $Y\in T_pM$, then $(d\exp{tX})_p(Y)$ is the parallel translate of Y along the geodesic. 
\end{enumerate}
\end{thm}

\subsubsection{Classification}
Any symmetric cone (i.e. homogeneous and self-dual) $\mathscr{P}$ is in a unique way isomorphic to the direct product of irreducible symmetric cones $\mathscr{P}_i$ (cf. \cite[Proposition III.4.5]{FK}). 
\medskip 

\begin{prop}~\label{P:Vclass}
Each irreducible homogeneous self--dual cone belongs to one
of the following classes:

\smallskip

\begin{table}[ht]
    \centering
    \begin{tabular}{|c|c|l|}
  \hline
Nb & Symbol & Irreducible symmetric cones \\
 \hline
 1. &      $ \mathscr{P}_n(\R)$ &  Cone of $n \times n$ positive definite symmetric real matrices. \\
     &  & \\
      
    2.  &      $ \mathscr{P}_n(\C)$ &  Cone of $n \times n$ positive definite self-adjoint complex matrices. \\
&  & \\
       3.  &    $ \mathscr{P}_n(\hH)$ &  Cone of $n \times n$ positive definite self-adjoint quaternionic matrices. \\
           &   & \\
        4. &   $ \mathscr{P}_3(\oO)$ & Cone of $3 \times 3$  positive definite self-adjoint octavic matrices. \\
           &   & \\
    5. &    $\Lambda_n$    & Lorentz cone  given by $x_0>\sqrt{\sum_{i=1}^n x_i^2}$ (spherical cone). \\
        &   & \\
       \hline
    \end{tabular}
    \caption{Irreducible symmetric cones}
    \label{tab:cones}
\end{table}
\end{prop}

 \subsubsection{Symmetric cones}\label{S:3.1}
The cones of positive definite quadratic forms are non-compact symmetric spaces. A {\it symmetric space} is a Riemannian space, which can be written as the quotient of Lie groups $G/K$, where $G$ is a connected Lie group with an involutive automorphism whose fixed point set is essentially the compact subgroup $K\subset G.$ 

The pair $(G,K)$ is a {\it symmetric pair} provided that there exists an involution $s\in G$, such that  $(K_{s})_{0} \subset K \subset K_{s}$, where $K_{s}$ is the set of fixed points of $s$ and $(K_{s})_{0}$ is the identity component of $K_{s}$. See  \cite[Chp.IV, paragraphs 3, 4 ]{Hel} for a detailed exposition.
\subsubsection{Classification Table}
One identifies the space of  $n\times n$ symmetric (resp. hermitian) positive definite matrices over a real division algebra with the following non-compact symmetric spaces $GL_n(\R)/O_n$ (resp. $GL_n(\C)/U_n$, $GL_n(\hH)/Sp_n$ $GL_3 (\oO)/F_4$) (see \cite[p.97]{FK}). This is summarised in the following classification table. 
\begin{table}[ht]
    \centering
\begin{tabular}{|c|c|c|c|c|c|}
     \hline
Cone  & $G/K$ & $T_q\mathscr{P}_n$  & $\frak{g}$ & $\frak{k}$  \\
 \hline
 $\mathscr{P}_n(\R)$ & $GL_n(\R)/O_n $& $Sym(n,\R)$ & $ \frak{sl}(n,\R)\oplus \R$ & $\frak{o}$(n) \\
 $\mathscr{P}_n(\C)$ &  $GL_n(\C)/U_n$ &$Herm(n,\C)$  &  $ \frak{sl}(n,\C)\oplus \R$ & $\frak{su}$(n)  \\
 $\mathscr{P}_n(\hH)$ & $GL_n(\hH)/Sp_n$ & $Herm(n,\hH)$ &  $\frak{sl}(m,\hH)\oplus \R$ & $\frak{su}(n,\hH$) \\
 $\mathscr{P}_3(\oO)$ & $GL_3 (\oO)/F_4$  &$ Herm(3,\oO)$ & $ \frak{e}_{(-26)}\oplus \R$ & $\frak{f}_4$ \\
  \hline
\end{tabular}
\caption{}\label{T:table1}
\end{table}

To clarify the notations:
\begin{itemize}
 \item $T_q\mathscr{P}_n$ is the tangent space to the cone at a point $q$ in $\mathscr{P}_n$.
\item  $Sym(n,\mathbb{K})$ stands for the space of $n\times n$ symmetric matrices defined over $\mathbb{K}$;
\item  $Herm(n,\mathbb{K})$ denotes the space of $n\times n$ self-adjoint matrices  defined over $\mathbb{K}$; 
 \item $\frak{g}$ is the Lie algebra associated to $G$;
 \item $\frak{k}$  is the Lie algebra associated to $K$;
\end{itemize}

The tangent space to $\mathscr{P}$ at a point carries a Jordan algebra structure. We recall this in the table \ref{tab:Jordan}: 
\begin{table}[ht]
    \centering
    \begin{tabular}{|c|c|c|c|}
     \hline
    &&& \\
 Irreducible  & Formally real simple  & $dim \cJ $ & rank $\cJ $  \\
 symmetric cone & Jordan algebras $\cJ$ & & \\
& && \\
 \hline
& && \\
$ \mathscr{P}_n(\R)$  &  Jordan algebra of $n\times n$ self-adjoint real matrices & $\frac{1}{2}n(n+1)$ & $ n$  \\

& && \\
    $ \mathscr{P}_n(\C)$      &  Jordan algebra of $n\times n$  self-adjoint complex matrices & $n^2$ & $ n$    \\
   & && \\
    
    $ \mathscr{P}_n(\hH)$     &   Jordan algebra of $n\times n$  self-adjoint quaternionic matrices & $n(2n-1)$ & $n$   \\
  & && \\
    
     $ \mathscr{P}_3(\oO)$    &  Jordan algebra of $3\times 3$ self-adjoint octonionic matrices: &   27 & 3   \\
       &  Albert algebra. &  &\\
      & && \\
      \hline
    \end{tabular}
    \caption{Simple formally real Jordan algebras}
    \label{tab:Jordan}
\end{table}

\subsubsection{Symmetric bilinear forms}
It is a well known fact that if $G$ is semi-simple, the Killing bilinear form is non-degenerate on $\fg$. The symmetric bilinear form is given by:  
$$\langle X,Y\rangle=-Tr(ad X\, ad Y),$$ where $ad\, Z(\xi)=[Z,\xi]$ and $Z \in \fg$. Therefore, we may state the following:

\begin{prop}
Let  $\mathscr{P}$ be a cone listed in Table \ref{T:table1}. Then, this irreducible cone comes equipped with a $G$-invariant metric and with a symmetric bilinear
form given by 
\begin{equation}\label{E:bili}
    \langle X,Y\rangle=\Re\, Tr(XY),
\end{equation} where $X,Y\in T_p(Gl_n(\K)/K)\cong \ft\subset \fg$,  where $Tr(\cdot)$ stands for the trace operator and $\Re$ is the real part.
\end{prop}
\begin{proof}
This statement follows from the existence of the Killing form. See \cite[p. 46, Proposition III.1.5]{FK} for a precise statement. 
\end{proof}

\section{Landau--Ginzburg model}\label{S:LG}
We compare this construction with the LG model. Mathematically the LG model is summarised as a non-compact K\"ahler manifold and a holomorphic Morse function.  To improve the mathematical understanding, we propose to recall additional information which are present in the original construction. 

\subsubsection{}
 
In 1957, Bardeen, Cooper, and Schrieffer (BCS) introduced the foundation for a quantum theory of superconductivity. This gave the BCS theory.
\begin{itemize}
    \item There exists an important object:  the (BCS) {\it wave function} $\psi$. It is a function of two-particle system $k \uparrow$ and -$k \downarrow$ ($\uparrow$ and $\downarrow$ stand for the spin). 
    \item In the BCS state, one-particle orbitals are occupied in pairs: if an orbital with wave vector $k$ and spin up is occupied, the orbital with wave vector $k$ and spin down is also occupied. If $k \uparrow$ is vacant, then so is $ k \downarrow$ vacant.
    \end{itemize}
The pairs are called {\it Cooper pairs} and they have spin zero. 

For a complete set of states of a two-electron system satisfying periodic boundary conditions in a cube of unit volume, take the plane wave product functions $\varphi(k_1,k_2;r_1,r_2)=\exp(\imath(k_1\cdot r_1+k_2\cdot r_2))$. It is assumed that the electrons are of opposite spin.
So, the plane wave can be expressed as $\varphi(K,k;R,r)=\exp(\imath K\cdot R)\exp(\imath k\cdot r),$ where we have:

 the barycentre $R$ of $r_1$ and $r_2$, $R=\frac{1}{2}(r_1+r_2)$ and we write $r=r_1-r_2$; Reciprocally,
$K=k_1+k_2$ and $k=\frac{1}{2}(k_1-k_2)$.

\subsubsection{}
Let us introduce the order parameter $\psi(r)$ such that  $\overline{\psi(r)}\psi(r)=|\psi(r)|^2=n(r)$. This is a complex valued function. It corresponds to the local concentration of superconducting electrons. It represents a condensed wave function which is a single quantum state occupied by a large fraction of Cooper pairs. In other words, this generates the {\it density of probability} of finding Cooper pairs in a given domain.

\subsubsection{}
 We explain how the holomorphic Morse function mentioned above is obtained. It corresponds to the free energy density, expressed as a function of $\psi(r)$ as follows: 

\begin{equation}
F_s(r)=F_0-\alpha |\psi|
^2 +\frac{\beta}{2}|\psi|^4 + \frac{1}{2m}|-\imath \hbar  \nabla- q\frac{A}{c}\psi |^2 - \int_0^{B_a} M \cdot dB_a
\end{equation}
where:
\begin{itemize}
\item $\alpha, \beta,m$ are positive constants;
\item $F_0$ is the free energy density of the normal state;
\item $- \alpha |\psi |
^2 +\frac{\beta}{2}|\psi|^4$ is a Landau form for the expansion of free energy vanishing at a second-order phase transition;
\item the term in $| \nabla |\psi |^2$ represents an increase in energy caused by a spatial variation of the order parameter. 
\item the term  $ \int_0^{B_a} M \cdot dB_a$ represents the increase in the superconducting free energy.
\end{itemize}
\subsubsection{Landau--Ginzburg equation}
In order to obtain the Landau--Ginzburg equation it is necessary to minimise the total free energy $\int F_s(r)dV$ w.r.t $\psi(r)$.
We get 
\[\delta F_s(r)=[-\alpha \psi + \beta |\psi |^2 \psi +\frac{1}{2m}(-\imath \hbar \nabla-q\frac{A}{c})\delta\overline{\psi}+c.c.]\]
If $\delta \overline{\psi}$ vanishes on the boundaries, it follows that 
$\int \delta F_s(r)dV=\int dV \delta\overline{\psi} [-\alpha \psi + \beta |\psi |^2\psi +\frac{1}{2m}(-\imath \hbar \nabla-q\frac{A}{c})]+c.c$.
The integral vanishes if the term in the bracket is null. That is if
\begin{equation}
    [\frac{1}{2m}(-\imath \hbar \nabla-q\frac{A}{c})^2 -\alpha \psi + \beta | \psi |^2]\psi=0
\end{equation} 
This last equation is precisely the LG equation. 

By minimising the free energy $F_s$ wrt $\delta A$ one gets a gauge-invariant expression for the super current flux:
\[j_s(r)=\frac{\imath g\hbar}{2m}[\overline{\psi},\nabla\psi] -\frac{q^2}{mc}|\psi|^2A.\]

\vfill\eject

\end{document}